\documentclass[12pt]{amsart}
\usepackage[margin=1in]{geometry}
\usepackage{amsthm,amsmath,times,amsfonts,amssymb}
\usepackage{dyck_paths}
\usepackage[norelsize]{algorithm2e}

\newtheorem{theorem}{Theorem}

\newtheorem{prop}[theorem]{Proposition}
\newtheorem{example}[theorem]{Example}

\newtheorem{cor}[theorem]{Corollary}

\theoremstyle{remark}
\newtheorem{remark}[theorem]{Remark}

\newcommand\Area{ \mathsf{Area} }
\newcommand\Dinv{ \mathsf{Dinv} }
\newcommand\Skips{ \mathsf{Skip} }
\newcommand\area{ \mathsf{area} }
\newcommand\dinv{ \mathsf{dinv} }
\newcommand\skips{ \mathsf{skip} }
\newcommand\ides{ \mathsf{iDes} }
\newcommand\leg{ \mathsf{leg} }
\newcommand\arm{ \mathsf{arm} }
\newcommand\bounce{ \mathsf{bounce} }

\def \R {\mathbb R}
\def \Q {\mathbb Q}

\def \PF {\mathsf{PF}}

\renewcommand{\th}{^{\text{th}}}

\newcommand*\rfrac[2]{{}^{#1}\!/_{#2}}

\title{Rational Catalan Polynomials and Rank Words}
\author{Ryan Kaliszewski and Huilan Li}

\begin{document}
\maketitle

\begin{abstract} For $m,n$ coprime we introduce a new statistic $\skips$ on $(m,n)$-rational Dyck paths and give a fast way to compute $\dinv$ and $\skips$ statistics. We also introduce $(m,n)$-rank words, which are in one-to-one correspondence with $(m,n)$-Dyck paths. Defining an equivalence relation on pairs of certain ranks in a rank word, we prove that the number of equivalence classes is the $\skips$ of the rank word, and the $\skips$ of the corresponding Dyck path. We construct a homogeneous generating function $W_{m,n}(q,t,b)$ using statistics $\area, \dinv$ and $\skips$, where $W_{m,n}(q,t,1)=C_{m,n}(q,t)$, the rational Catalan polynomial. We then give an explicit formula for $(3,n)$-rational Catalan polynomials and prove they are $q,t$-symmetric.
\end{abstract}

\section{Introduction}

In the early 1990's Garsia and Haiman introduced an important sum of
functions in
$\Q(q,t)$, the $q,t$-Catalan polynomial $C_n(q,t)$,
which has since been shown to have interpretations in terms of algebraic
geometry and representation theory.
These classical $q,t$-Catalan polynomials are given by
\begin{align} C_n(q,t) & =\sum_{\pi}q^{\dinv(\pi)}t^{\area(\pi)} \\  \label{pt2}  &=\sum_{\pi}q^{\area(\pi)}t^{\bounce(\pi)},\end{align}
where the sums are over all Dyck paths $\pi$ from $(0,0)$ to $(n,n)$, with \eqref{pt2} due to \cite{Haglund03}.

It was proven that 
\[ C_n(q,t)=C_n(t,q) \]
algebraically; however, an involution on the set of Dyck paths that exchanges  $(\dinv,\area)$ or $(\area,\bounce)$ has yet to be discovered.  For an overview of the classical $q,t$-Catalan polynomials and Dyck paths, see \cite{Garsia96, Garsia02, Haglund05,Haglund08}.

Recently, a valuable generalization of the classic $q,t$-Catalan polynomial has come to light \cite{Hikita12}.  For positive integers $m,n$ that are coprime, these \emph{$(m,n)$-rational} $q,t$-Catalan polynomials have a similar description to the classic case
\begin{equation} C_{m,n}(q,t) =\sum_{\pi}q^{\dinv(\pi)}t^{\area(\pi)}, \end{equation}
where the sum is over all \emph{rational} Dyck paths $\pi$ from $(0,0)$ to $(m,n)$.

A rational Dyck path is a path in the $m\times n$ lattice which proceeds by north and east steps from $(0,0)$ to $(m,n)$ and which always remains weakly above the main diagonal $y=\frac{m}{n}x$.  The collection of cells lying above a Dyck path $\pi$ always forms an English Ferrers diagram $\lambda(\pi).$ 

The classic $q,t$-Catalan polynomials correspond to the cases $m=n+1$. It is an open conjecture whether the $(m,n)$-rational $q, t$-Catalan polynomials are symmetric in $q$ and $t$.  Gorsky and Mazin have proven the $q,t$-symmetry for the case when $m\leq 3$ without giving an explicit formula for the polynomials \cite{Gorsky14}.  Independently, Lee, Li and Loehr also have proven the $q,t$-symmetry when $m\leq 4$ \cite{LLL14}. In 2015, Kaliszewski and Li introduced a new statistic, $\skips$, on rational $(3,n)$-Dyck paths, proved the $q,t$-symmetry  and gave an explicit formula of $(m,n)$-rational $q, t$-Catalan polynomials for $m=3$ \cite{Kalis15}.

Also in \cite{Hikita12}, Hikita extended parking functions and their statistics to the $(m, n)$-rational case and defined polynomials 
\begin{equation} \label{hikitapoly} H_{m,n}(X;q,t)=\sum_{\PF}t^{\area(\PF)}q^{\dinv(\PF)}F_{\ides(\PF)}(X), \end{equation}
where the sum is over all parking functions $\PF$ over all $(m,n)$-Dyck paths and the $F$ are the (Gessel) fundamental quasisymmetric functions indexed by the inverse descent composition of the reading word of the parking functions $\PF$. When $m=2$ or $n=2$, the structure  of Hikita's polynomials  have been completely described by \cite{Leven14}.

Since, for each $(m,n)$-rational Dyck path there is a parking function on that path with the same statistics as the underlying Dyck path, the $(m,n)$-rational $q,t$-Catalan polynomial appears within $H_{m,n}(X;q,t)$.  Specifically, $C_{m,n}(q,t)$ is the coefficient of $F_{(1^n)}$ in the expansion of the Hikita's polynomial $H_{m,n}(X;q,t)$.  Thus, describing the structure of the $(m,n)$-rational $q,t$-Catalan polynomials is the first step to uncovering the structure of Hikita's polynomials.

In this paper we introduce a new set of objects, $(m,n)$-rank words and $\area, \dinv$ and $\skips$ statistics on the $(m,n)$-rank words.  We construct the homogeneous generating function
\begin{equation} W_{m,n}(b,q,t)=\sum_w b^{\skips(w)}q^{\dinv(w)}t^{\area(w)}, \end{equation}
where the sum is over all $(m,n)$-rank words $w$.  Then we prove two useful theorems:

\begin{theorem} If $m,n$ are positive integers that are coprime then
\[ W_{m,n}(1,q,t)=C_{m,n}(q,t). \]
\end{theorem}
\begin{theorem} If $n$ is a positive integer that is not divisible by $3$ then
\[ W_{3,n}(b,q,t)=\sum_{i=0}^{\lfloor n/3 \rfloor} b^i\cdot s_{n-1-2i,i}(q,t), \]
where $\{s_\lambda\}_\lambda$ is the Schur basis for symmetric functions.
\end{theorem}

This not only proves that the $(3,n)$-rational $q,t$-Catalan polynomials are $q,t$-symmetric, but also completely describes them:
\begin{cor}
\[ C_{3,n}(q,t)=\sum_{i=0}^{\lfloor n/3 \rfloor} s_{n-1-2i,i}(q,t).\]
\end{cor}

\subsection{Rational Dyck Paths}

Suppose that $m,n$ are positive coprime integers.  Construct the \emph{$(m,n)$-lattice} by drawing a rectangular integer lattice in $\R^2$ whose southwest corner lies on the origin and whose northeast corner lies on the point $(m,n)$.  The cell whose northeast corner lies on the point $(u,v)$ will be referred to as cell $(u,v)$.  Thus we can define the $i\th$ row as the set of cells
\[ \{(u,i)|1\leq u\leq n\} \]
and the $j\th$ column as the set of cells
\[ \{(j,v)|1\leq v\leq m\}. \]
The \emph{$(m,n)$-diagram} is the $(m,n)$-lattice where each cell $(u,v)$ contains an integer $a$ that satisfies
\begin{equation} a = mn-un-(n+1-v)m. \end{equation}
We call $a$ the \emph{rank} of cell $(u,v)$, denoted by $\gamma_{m,n}(u,v)$ or simply $\gamma(u,v)$ when the parameters are clear from context.  Since $m$ and $n$ are coprime, there are no duplicate ranks.

An \emph{$(m,n)$-rational Dyck path} or $(m,n)$-Dyck path is path on the $(m,n)$-diagram that begins at $(0,0)$ and ends at $(m,n)$.  The path can only consist of northward and eastward steps and must always lie above the diagonal $y=\frac{n}{m}x$.

An $(m,n)$-Dyck path partitions the cells within the $(m,n)$-diagram into two sets.  Since the path must lie above the diagonal, one of the sets will always contain the southeast corner of the $(m,n)$-diagram.  We say that any of the cells in this set are \emph{below the path}.  The other set of cells, which may be empty, are \emph{above the path}.  A rank is \emph{above the path} if it is in a cell that is above the path, or \emph{below the path} if it is in a cell that is below the path.

\begin{example} A $(4,7)$-Dyck path:
\[ \pi=\dyckpath{4}{7}{0,0,0,0,0,2,2}{3cm}\;.\]
The cells that are above the path are colored gray and those that are below the path are colored cyan.  Rank 6 is above the path, while rank 5 is below the path.

Here $\gamma(1,5)=9$ and $\gamma(4,3)=-20$.

\end{example}

The set of cells above the path always has the shape of a Ferrers diagram (in English notation).  This is because the only allowed moves are northward steps and eastward steps, so the number of cells above the path in each row must be weakly increasing from bottom to top.

\subsection{Fast $\dinv$ and $\skips$ on Dyck paths}

 Using both the path and ranks we can improve the computation of the $\dinv$ statistic and we will define a new statistic, $\skips$,  on $(m,n)$-Dyck paths.

Let $\pi$ be an $(m,n)$-Dyck path.  We partition the cells containing positive ranks into three sets, $\Area(\pi),\Dinv(\pi)$ and $\Skips(\pi)$.  Define 
\begin{equation} \Area(\pi)=\{(x,y):(x,y)\text{ is below path } \pi \text{ and contains a positive rank}\}. \end{equation}
For the cells above the path define
\begin{equation} \label{dinv} \Dinv(\pi)=\left\{(x,y):\frac{\arm(x,y)}{\leg(x,y)+1}<\frac{m}{n}<\frac{\arm(x,y)+1}{\leg(x,y)}\right\}, \end{equation}
where $\arm(x,y)$ is the number of cells above $\pi$ and strictly east of $(x,y)$ and $\leg(x,y)$ is the number of cells above $\pi$ and strictly south of $(x,y)$.  We interpret division by zero to give infinity.  Define the $\Skips$ set to be the remaining cells above the path:
\begin{equation} \Skips(\pi)=\{(x,y):(x,y)\not\in\Dinv(\pi) \}\;. \end{equation}
Set
\[ \area(\pi)=|\Area(\pi)|,\qquad \dinv(\pi)=|\Dinv(\pi)|, \qquad \skips(\pi)=|\Skips(\pi)|. \]
The idea is that when counting the cells that contribute to $\dinv$ we should \emph{skip} any that do not satisfy the inequality in~\eqref{dinv}.

The \emph{fast $\dinv$} algorithm gives a faster and simpler way to determine whether or not a cell is in $\Dinv$.  For a given cell $(x,y)$ above an $(m,n)$-Dyck path $\pi$, let $(x,y)^\downarrow$ be the southernmost cell in the same column as $(x,y)$ that is above the path and $(x,y)^\Downarrow$ be the northermost cell in the same column as $(x,y)$ that is below the path.  Similarly, let $(x,y)^\rightarrow$ be the easternmost cell in the same row as $(x,y)$ that is above the path and $(x,y)^\Rightarrow$ be the westernmost cell in the same row as $(x,y)$ that is below the path.  So $(x,y)^\Rightarrow$ is exactly one cell to the east of $(x,y)^\rightarrow$ and $(x,y)^\Downarrow$ is exactly one cell south of $(x,y)^\downarrow$.

\begin{theorem}[Fast $\dinv$] \label{fastdinvthm}
Suppose $\pi$ is an $(m,n)$-Dyck path and let $(x,y)$ be a cell above the path in $\pi$.  The cell $(x,y)$ is in $\Dinv(\pi)$ if and only if
\begin{equation} \label{fastdinv} \gamma[(x,y)^\rightarrow]>\gamma[(x,y)^\Downarrow] \qquad \text{and}\qquad \gamma[(x,y)^\downarrow]>\gamma[(x,y)^\Rightarrow]. \end{equation}
\end{theorem}

\begin{example} Consider cell $(1,7)$ of the $(5,7)$-Dyck path
\[ \pi=\dyckpath[(1,7)]{5}{7}{0,0,1,1,1,1,2}{3cm}. \]
We compute and get
$ \gamma[(1,7)^\rightarrow]=16>\gamma[(1,7)^\Downarrow]=-2\;. $
\[ \pi=\dyckpath[(2,7),(1,2)]{5}{7}{0,0,1,1,1,1,2}{3cm}. \]
But $ \gamma[(1,7)^\downarrow]=3<\gamma[(1,7)^\Rightarrow]=9\;. $
\[ \pi=\dyckpath[(3,7),(1,3)]{5}{7}{0,0,1,1,1,1,2}{3cm}. \]
So $(1,7)\not\in\Dinv(\pi)$.
\end{example}

\begin{proof} Let $\pi$ be an $(m,n)$-Dyck path.  Observe that for any cell $(x,y)$ above the path
\[ \gamma(x,y)-\gamma[(x,y)^\rightarrow]=n\cdot\arm(x,y),\qquad\gamma(x,y)-\gamma[(x,y)^\Rightarrow]=n\cdot(\arm(x,y)+1) \]
and
\[ \gamma(x,y)-\gamma[(x,y)^\downarrow]=m\cdot\leg(x,y),\qquad\gamma(x,y)-\gamma[(x,y)^\Downarrow]=m\cdot(\leg(x,y)+1)\;. \]
So
\begin{align*} \gamma[(x,y)^\rightarrow]<\gamma[(x,y)^\Downarrow] & \iff \gamma(x,y)-\gamma[(x,y)^\Downarrow]<\gamma(x,y)-\gamma[(x,y)^\rightarrow] \\
  & \iff m\cdot(\leg(x,y)+1)<n\cdot\arm(x,y) \\
  & \iff \frac{m}{n}<\frac{\arm(x,y)}{\leg(x,y)+1}.
\end{align*}
Similarly,
\begin{align*} \gamma[(x,y)^\downarrow]<\gamma[(x,y)^\Rightarrow] & \iff \gamma(x,y)-\gamma[(x,y)^\Rightarrow]<\gamma(x,y)-\gamma[(x,y)^\Downarrow] \\
  & \iff n\cdot(\arm(x,y)+1)<m\cdot\leg(x,y) \\
  & \iff \frac{\arm(x,y)+1}{\leg(x,y)}<\frac{m}{n}.
\end{align*}
Therefore inequalities \eqref{fastdinv} and the inequality in \eqref{dinv} are equivalent.
\end{proof}

\begin{cor} \label{fasterdinv} Suppose that $\pi$ is an $(m,n)$-Dyck path and let $(x,y)$ be a cell above path $\pi$.  Suppose $a$ is a rank in the same row as $(x,y)$ and $b$ is a rank in the same column as $(x,y)$.  If
\[ b<a, \text{ with $a$ below the path and $b$ above the path} \]
or
\[ a<b, \text{ with $b$ below the path and $a$ above the path} \]
then $(x,y)\not\in\Dinv(\pi)$.
\end{cor}

\begin{proof} Suppose $a$ is a rank in the same row as $(x,y)$ and $b$ is a rank in the same column as $(x,y)$. 

First assume that $b<a$ with $a$ below the path and $b$ above the path.  So $a\leq\gamma[(x,y)^\Rightarrow]$ and $b\geq\gamma[(x,y)^\downarrow]$.  Therefore,
\[ \gamma[(x,y)^\downarrow]\leq b<a \leq\gamma[(x,y)^\Rightarrow] \]
and the result follows from Theorem \ref{fastdinvthm}.

Similarly, if we assume that $a<b$ with $b$ below the path and $a$ above the path, then $b\leq\gamma[(x,y)^\Downarrow]$ and $a\geq\gamma[(x,y)^\rightarrow]$.  Therefore,
\[ \gamma[(x,y)^\rightarrow]\leq a<b \leq\gamma[(x,y)^\Downarrow] \]
and the result follows from Theorem \ref{fastdinvthm}.
\end{proof}

\begin{cor} \label{rightcol} Suppose that $\pi$ is an $(m,n)$-Dyck path.  If $(m-1,v)$ is a cell that is above the path then $(m-1,v)\in\Dinv(\pi)$.
\end{cor}
\begin{proof} Suppose $(m-1,v)$ is a cell above the path $\pi$. If $a$ is a rank in column $m$ then $a=mn-mn-km=-km$, so $a<0$.  Therefore
\[ \gamma[(m-1,v)^\downarrow] >0>\gamma(m,v)=\gamma[(m-1,v)^\Rightarrow]\;. \]
Also $(m-1,v)^\rightarrow=(m-1,v)$, so
\[ \gamma[(m-1,v)^\rightarrow]=\gamma(m-1,v)\geq \gamma[(m-1,v)^\downarrow]>\gamma[(m-1,v)^\Downarrow]. \]
So $(m-1,v)\in\Dinv(\pi).$
\end{proof}

\section{Rank Words}

Given a pair of positive coprime integers, we introduce a new set of combinatorial objects, \emph{rank words}.  We show that rank words are in correspondence with rational Dyck paths. So they are enumerated by the rational Catalan numbers.  Then we introduce a triple of statistics $\area,\dinv,\skips$ on rank words and construct a generating function.

Suppose that $m,n$ are positive coprime integers.  The \emph{$(m,n)$-word} is the word consisting of all positive integers that can be written as
\begin{equation} \label{mnword}   mn-km-\ell n,  \end{equation}
where $k,\ell\geq 1$.  We write the $(m,n)$-word in increaing order and we decorate each number $a$ with subscript $\ell$ from Equation \eqref{mnword}.  When we write $a_\ell$ in an $(m,n)$-word we refer to $a$ as the \emph{rank} and $a_\ell$ as the \emph{letter} with \emph{color} $\ell$.

\begin{example}
Suppose $m=3,n=5$.  We can compute the ranks in the word:
\[ 7 = 3\cdot 5 - {\color{blue} 1}\cdot 3-{\color{red} 1}\cdot5  \]
\[ 4 = 3\cdot 5 - {\color{blue} 2}\cdot 3-{\color{red} 1}\cdot 5 \]
\[ 2 = 3\cdot 5 - {\color{blue} 1}\cdot 3-{\color{red} 2}\cdot 5 \]
\[ 1 = 3\cdot 5 - {\color{blue} 3}\cdot 3-{\color{red} 1}\cdot 5 \]

So the $(3,5)$-word is 
\[ \rankword{3}{5}{0,5,10}{4cm}. \]
\end{example}

An \emph{$(m,n)$-rank word} is an $(m,n)$-word with a subset of the letters highlighted.  The highlighting has to follow the rule:  if $a_\ell = mn-km-\ell n$ is highlighted then for all $i=1,\ldots,k$ and $j=1,\ldots,\ell$
\[ mn-im-jn \]
is highlighted.  Equivalently, if a rank $a$ is highlighted and $b>a$, with
\[ b\equiv a\pmod n \]
or
\[ b\equiv a\pmod m \]
then $b$ is also highlighted.

\begin{example} The $(3,5)$-rank words are
\[ \rankword{3}{5}{0,5,10}{4cm} \qquad\qquad \rankword{3}{5}{0,5,7}{4cm}  \]
\[ \rankword{3}{5}{0,5,4}{4cm} \qquad\qquad \rankword{3}{5}{0,5,1}{4cm}  \]
\[ \rankword{3}{5}{0,2,7}{4cm} \qquad\qquad \rankword{3}{5}{0,2,4}{4cm}  \]
\[ \rankword{3}{5}{0,2,1}{4cm}\]
\end{example}

\begin{remark} If $a_\ell$ and $b_\ell$ are letters with the same color then $a\equiv b\pmod m$.  This is because $\exists k_1,k_2$ such that
\[ a=mn-k_1m-\ell n \qquad\text{and}\qquad b=mn-k_2m-\ell n \]
so
\[ a-b=(k_1-k_2)m. \]
Therefore if $a<b$ and $a_\ell$ is highlighted then so is $b_\ell$.
\end{remark}

Notice that in the previous example the highlighted word
\[ \rankword{3}{5}{0,2,10}{4cm} \]
is missing.  Because $2\equiv 7\pmod 5$, if $2_2$ is highlighted then $7_1$ \emph{must} be highlighted.

\subsection{Correspondence of $(m,n)$-Rank Words to $(m,n)$-Dyck Paths}

Let $m,n$ be positive coprime integers and $w$ be the $(m,n)$-word.  Construct the \emph{$(m,n)$-diagram} $D$ and note that the set of positive ranks appearing in $D$ is precisely the set of ranks in $w$.  Furthermore, the color $\ell$ of a letter $a_\ell$ in $w$ is the column in which the rank $a$ appears in $D$.

\begin{example}
Let $m=4,n=7$.  The $(4,7)$-word is
\[ \rankword{4}{7}{0,7,14,21}{8cm} \]
and the $(4,7)$-diagram is
\[ \rankdiagram{4}{7}{3cm} \;. \]
We have drawn the diagonal to show that all of the positive ranks are in cells that lie strictly above the diagonal.
\end{example}

\begin{prop} \label{wordsanddiagrams} Let $m,n$ be positive coprime integers.  The following are equivalent:
\begin{enumerate} 
  \item Cell $(u,v)$ lies strictly above the diagonal $y=\frac{n}{m}x$ in the $(m,n)$-diagram.
  \item A rank $a>0$ appears in cell $(u,v)$ of the $(m,n)$-diagram.
  \item A letter $a_u$ appears in the $(m,n)$-word.
\end{enumerate}
\end{prop}

\begin{proof} 
(1) $\Rightarrow$ (2) \hspace{1mm} Assume that cell $(u,v)$ containing rank $a$ lies strictly above the diagonal $y=\frac{n}{m}x$ in the $(m,n)$-diagram.  Recall that $(u,v)$ are the coordinates of the northeast corner of cell $(u,v)$.  In order for cell $(u,v)$ to lie strictly above the diagonal the southeast corner must lie strictly above the diagonal.  Therefore $v-1>\frac{n}{m}u$, or equivalently
\[ (v-1)m-un>0. \]
By adding and subtracting $mn$ we have
\begin{align*} 0 & < (v-1)m-un \\
  & =mn-mn+(v-1)m-un \\
  & =mn-mn-(1-v)m-un \\
  & =mn-(n+1-v)m-un \\
  & = a \;.
\end{align*}

(2) $\Rightarrow$ (3) \hspace{1mm} Suppose that $a>0$ is a rank appearing in cell $(u,v)$ of the $(m,n)$-diagram.  Then
\[ a=mn-(n+1-v)m-un. \]
Since $1\leq v\leq n$, we have $n+1-v\geq 1$.  So set $k=n+1-v$ and $\ell=u$ then we have
\[ a=mn-km-\ell n \]
where $k,l \geq 1$.  Therefore $a_\ell =a_u$ appears in the $(m,n)$-word.

(3) $\Rightarrow$ (2) $\Rightarrow$ (1) \hspace{1mm} Assume that letter $a_u$ appears in the $(m,n)$-word.  So 
\[ a=mn-km-un >0 \]
for some integer $k\geq 1$.  Set $v=n+1-k$ and since $1\leq k\leq n$ it follows that $1\leq v\leq n$.  Therefore cell $(u,v)$ lies within the $(m,n)$-diagram and contains rank $a$.  Since $a>0,$ then $mn-(n+1-v)m-un>0$.  By combining the first two terms we have that
\[ (v-1)m-un>0,\]
which is equivalent to cell $(u,v)$ lying strictly above the diagonal $y=\frac{n}{m}x$, as stated in the first part of the proof.
\end{proof}

Suppose that $w$ is a $(m,n)$-rank word and $D$ is the $(m,n)$-diagram.  For each highlighted letter $a_\ell$ in $w$ we highlight rank $a$ in the diagram $D$.  Since letters in $w$ have strictly positive rank, we are only highlighting cells with positive ranks in $D$.  We will show that for any rank word $w$ if a rank in $D$ is highlighted, so is every rank weakly northwest.

\begin{prop} \label{highlightedrank} Let $m,n$ be positive, coprime integers.  Suppose that $w$ is a $(m,n)$-rank word and $D$ is the corresponding $(m,n)$-diagram where a rank $a$ is highlighted in $D$ if and only if it is highlighted in $w$.  If the rank in cell $(u,v)$ of $D$ is highlighted and $(x,y)$ is a cell weakly northwest of $(u,v)$, then the rank in cell $(x,y)$ is also highlighted.
\end{prop}

\begin{proof} Assume that $w$ is an $(m,n)$-rank word and $D$ is the $(m,n)$-diagram where each highlighted rank in $w$ is also highlighted in $D$.  Suppose that $(u,v)$ is a cell of $D$ containing a highlighted rank and $(x,y)$ is a cell weakly northwest of $(u,v)$, i.e. $x\leq u$ and $y\geq v$.  

Since $y\geq v$ it must follow that $(n+1-y)\leq (n+1-v)$.  Similarly, $x\leq u$ gives $xn\leq un$.
By assumption,
\[ \gamma(u,v)=mn-(n+1-v)m-un \]
is highlighted in $D$, and thus also highlighted in $w$.  But by the definition of $(m,n)$-rank words, since $D(u,v)$ is highlighted in $w$ and
\[ \gamma(x,y)=mn-(n+1-y)m-xn \]
with $n+1-y\leq n+1-v$ and $x\leq u$ then $\gamma(x,y)\geq \gamma(u,v)>0$ must appear in $w$ and be highlighted.  So $\gamma(x,y)$ is highlighted in $D$.
\end{proof}

Given a highlighted $(m,n)$-diagram $D$, if a cell contains a highlighted rank, then every cell directly west and directly north must also contain a highlighted rank. So the cells in $D$ with highlighted ranks form a Ferrers diagram (in English notation) with a partition shape.  To construct an $(m,n)$-Dyck path from $(0,0)$ to $(m,n)$, we draw the edges so that every highlighted rank is strictly above the path and every rank not highlighted is strictly below the path. The path is an $(m,n)$-Dyck path.

\begin{theorem} Suppose that positive integers $m,n$ are coprime.  The set of $(m,n)$-Dyck paths is in one-to-one correspondence with the set of $(m,n)$-rank words.
\end{theorem}
\begin{proof}  Propositions \ref{wordsanddiagrams} and \ref{highlightedrank} show that $(m,n)$-rank words correspond to $(m,n)$-diagrams with a subset of positive ranks highlighted in a partition pattern.  The argument above shows that these highlighted $(m,n)$-diagrams correspond to $(m,n)$-Dyck paths.
\end{proof}

If $w$ is an $(m,n)$-rank word we will denote the corresponding highlighted $(m,n)$-diagram by $D(w)$ and the corresponding $(m,n)$-Dyck path by $\Pi(w)$.

\begin{example} \label{firstdyck} Consider the $(4,7)$-rank word
\[ w=\rankword{4}{7}{0,7,6,13}{8cm} \]
The corresponding highlighted $(4,7)$-diagram is
\[ D(w)=\rankdiagram[(1,7),(2,7),(1,6),(2,6)]{4}{7}{3cm}\;. \]
This corresponds to the $(4,7)$-Dyck path:
\[ \Pi(w)=\dyckpath{4}{7}{0,0,0,0,0,2,2}{3cm}\;. \]
\end{example}

\subsection{The $\area,\dinv$ and $\skips$ on $(m,n)$-rank words}

We define statistics $\area, \dinv,$ and $\skips$ on $(m,n)$-rank words.  We show that these statistics correspond directly to their counterparts on $(m,n)$-Dyck paths. That is, if $w$ is an $(m,n)$-rank word then
\[ \area(w)=\area(\Pi(w)), \qquad \dinv(w)=\dinv(\Pi(w)), \qquad\text{and}\qquad \skips(w)=\skips(\Pi(w))\;.\]

Let $w$ be an $(m,n)$-rank word.  Define the $\area$ statistic on $w$ to be
\[ \area(w)=\{a_k\in w: a_k \text{ is not highlighted}\}\;. \]

\begin{prop}
If $w$ is an $(m,n)$-rank word 
\[ \area(w)=\area(\Pi(w))\;.\]
\end{prop}

\begin{proof} The letters in $w$ that are not highlighted correspond to the cells with positive ranks that are below path $\Pi(w)$.
\end{proof}

Let $S$ be the set of ordered pairs of letters $(a_k,b_\ell)$ where $b_\ell$ is not highlighted and $a_k$ is highlighted and left of $b_\ell$ in the rank word.  Partition $S$ into $S_>$ and $S_<$ by
\[ S_<=\{(a_k,b_\ell)\in S:k<\ell\} \qquad\text{and}\qquad S_>=\{(a_k,b_\ell)\in S:k>\ell\}. \]
By the restrictions on the highlighting in $w$ there will never be an ordered pair $(a_k,b_k)$ in $S$.

Define an equivalence relation on $S=S_<\cup S_>$ by
\begin{equation} \label{equivrel} (a_k,b_\ell)\sim(c_r,d_t) \qquad\text{ if } \qquad
\left\{\begin{array}{c}k=r \text{ and } b\equiv d \pmod n \\
\text{or}\\
\ell=t \text{ and } a\equiv c \pmod n\end{array} \right.\;.  \end{equation}
That is, $a$ and $c$ are in the same column and $b$ and $d$ are in the same row, or $a$ and $c$ are in the same row and $b$ and $d$ are in the same column.

\begin{prop} If $(a_k,b_\ell)\sim(c_r,d_t)$ in $S$, then both $(a_k,b_\ell),(c_r,d_t)$ are in $S_<$ or they are both in $S_>$.  
\end{prop}

\begin{proof} Suppose $(a_k,b_\ell)\sim(c_r,d_t)$ with $k<\ell$ and $r>t$. 

Case 1: Suppose $k=r$ and $b\equiv d \pmod n$, i.e., $a$ and $c$ are in the same column and $b$ and $d$ are in the same row. We know $a=mn-p_1m-kn$ and $b=mn-p_2m-\ell n$ for some $p_1$ and $p_2$.  Since $k<\ell$ and $a<b$, it must follow that $p_1>p_2$.  So in the $(m,n)$-diagram, rank $b$ must appear in a row above the row containing rank $a$.  Since $k=r>t$, rank $d$ lies in the column left of column $k$. By assumption, rank $d$ is in the same row as rank $b$. So $d$ is in the cell northwest to the cell containing $a$. Rank $a$ is above the path, therefore rank $d$ must also be above the path, i.e., highlighted. This is a contradiction.

Case 2: Suppose $\ell=t$ and $a\equiv c \pmod n$, i.e., $a$ and $c$ are in the same row and $b$ and $d$ are in the same column. Same as above, from $k<\ell$ and $a<b$, in the $(m,n)$-diagram rank $b$ must appear in a row above the row containing rank $a$. So rank $b$ is in a row above rank $c$. By assumption $\ell=t<r$, rank $b$ lies in the column left to rank $c$. So rank $b$ is northwest to rank $c$. Rank $c$ is above the path, therefore rank $b$ must also be above the path, i.e., highlighted.  This is a contradiction.
\end{proof}

Let the equivalence class of $S$ containing the ordered pair $(a_k,b_\ell)$ be denoted $[a_k,b_\ell]$.  
Define
\[ \skips(w)=\left|\rfrac{S}{\sim}\right|\;\]
to be the number of equivalence classes.

\begin{example} Consider the $(5,7)$-rank word
\[ w=\rankword{5}{7}{0,7,9,16,3}{8cm}, \]
we have
\[ S=\{\;(3_1,4_3),(3_1,6_2),(3_1,11_2),(8_1,11_2),(9_3,11_2)\;\}\;.\]

Since $11\equiv 4\pmod 7$ the equivalence relation partitions $S$ into the following equivalence classes:
\[  [3_1,4_3]=\{(3_1,4_3),(3_1,11_2),(8_1,11_2)\},\qquad[3_1,6_2]=\{(3_1,6_2)\},\qquad[9_3,11_2]=\{(9_3,11_2)\}. \]
The next theorem will show that $\skips(\Pi(w))=\skips(w)=3$.
\end{example}

\begin{theorem}
If $w$ is an $(m,n)$-rank word then
\[ \skips(w)=\skips(\Pi(w)). \]
\end{theorem}

\begin{proof}
Let $w$ be an $(m,n)$-rank word and $\pi=\Pi(w)$ be the corresponding $(m,n)$-Dyck path.  Define a map $\psi$ from the equivalence classes of $S$ to the set of letters in $w$ by \[ \psi[a_k,b_\ell]=q_p, \]
where
\[ p=k \text{ and } b\equiv q \pmod n \qquad \text{if}\qquad k<\ell \]
or
\[ p=\ell \text{ and } a\equiv q \pmod n \qquad \text{if}\qquad k>\ell. \]
That is, if $k<\ell$ then $q$ is the rank in the same column as $a$ and the same row as $b$ and if $k>\ell$ then $q$ is the rank in the same row as $a$ and the same column as $b$.  Hence $\psi[a_k,b_\ell]$ is a unique rank that is either directly above or directly to the left of $a_k$, so it is above the path.

Let $f$ be the function that associates to $q_p$ the cell in $\pi$ that contains rank $q$.  Since the cells in $\pi$ containing positive ranks correspond uniquely to the letters appearing in $w$, $f$ is a bijection.  Define 
\[ \varphi:\left|\rfrac{S}{\sim}\right|\rightarrow\Skips(\pi) \qquad\text{by}\qquad \varphi = f\circ\psi\;. \]

Since $b$ is a rank in the same row or column of $q$ and is below the path with $a<b$,  by Corollary \ref{fasterdinv}, the cell containing $q$ is in $\Skips(\pi)$.  So the image of $\varphi$ is in $\Skips(\pi)$.  Thus $\varphi$ is well-defined.

Next we will prove that $\varphi$ is injective.  Suppose that 
$[a_k,b_\ell], [c_r,d_t]$ are equivalence classes in $\rfrac{S}{\sim}$ such that
\[ \psi[a_k,b_\ell]=\psi[c_r,d_t]. \]
If $k<\ell$ and $r>t$ then
\[ k=t \qquad\text{ and }\qquad b\equiv c \pmod n. \]
Thus the ranks $a,d$ appear in the same column and the ranks $b,c$ appear in the same row.  Since $a,c$ are above the path and $b,d$ are below the path, it must follow that 
\[ a>d \qquad \text{and}\qquad c>b. \]
Therefore,
\[ a>d>c>b>a,\]
which is a contradiction.

If $k<\ell$ and $r<t$ then $k=r$ and $b,d$ are equivalent modulo $n$.  Therefore $(a_k,b_\ell)\sim(c_r,d_t)$.  The result is similar for $k>\ell$ and $r>t$.  So $\psi$ is injective and since $f$ is a bijection, $\varphi$ is injective.

Lastly, we will prove that $\varphi$ is surjective.  Suppose that $(x,y)\in\Skips(\pi)$. Set
\[ a=\pi[(x,y)^\rightarrow]\;, \qquad c=\pi[(x,y)^\downarrow], \]
\[ b=\pi[(x,y)^\Downarrow]\;, \qquad d=\pi[(x,y)^\Rightarrow]. \]
Since $(x,y)\in\Skips(\pi),$ then either $a<b$ or $c<d$.  Let $k$ be the column containing rank $a$.  Since $a$ and $c$ are above the path, $a_k$ and $c_y$ are highlighted in $w$.  Since $b$ and $d$ are below the path, $b_y$ and $d_{k+1}$  are not highlighted in $w$.
If $a<b$, since $a$ is in the same row of $(x,y)$ we have
\[ \psi[a_k,b_y]=q_y, \]
where $q=\gamma(x,y)$.  If $c<d$, since $c$ is in the same column of $(x,y)$ we have 
\[ \varphi[c_y,d_{k+1}]=q_y,\]
where $q=\gamma(x,y)$.  Hence $\psi$ is surjective, and since $f$ is bijective then $\varphi$ is a bijection.

Since $\varphi$ is a bijection it follows that $\skips(w)=\skips(\pi)$.
\end{proof}

Note that in an $(m,n)$-diagram $D$, each of the ranks in the rightmost column is negative as well as each rank in the bottom row.  If we disregard the rightmost column and bottom row, we see that there are $(m-1)(n-1)$ ranks remaining in the diagram.  If $\gamma(u,v)=a$ then 
\[ a=mn-un-(n-v+1)m.  \] 
Consider, 
\begin{align*} \gamma(m-u,n-v+2) & =mn-(m-u)n-[n-(n-v+2)+1]m \\
 & =un+(-v+1)m \\
 & =-mn+un+(n-v+1)m\\
 & = -a, \end{align*}
and $0$ does not appear in the set of ranks.  Therefore, half of the ranks are positive and half are negative.  Thus there are $\frac{(n-1)(m-1)}{2}$ positive ranks.

Define
\[ \dinv(w)=\frac{(n-1)(m-1)}{2}-\area(w)-\skips(w) \]
and the fact 
\[ \dinv(w)=\dinv(\Pi(w)) \]
follows from the fact that there are $\frac{(n-1)(m-1)}{2}$ positive ranks.

\subsection{Generating Functions of $(m,n)$-Rank Words}

Let $m,n$ be positive coprime integers.  Define the generating function
\begin{equation} W_{m,n}(b,q,t)=\sum_w b^{\skips(w)}q^{\dinv(w)}t^{\area(w)} \end{equation}
where the sum is over all $(m,n)$-rank words $w$.  Since for fixed $m,n$ the number of positive ranks is $\frac{(n-1)(m-1)}{2}$, the generating function is homogeneous of degree $\frac{(n-1)(m-1)}{2}$.

\begin{theorem} \label{wordtocatalan} For positive coprime integers $m,n$,
\[ W_{m,n}(1,q,t)=C_{m,n}(q,t) \]
\end{theorem}
\begin{proof}
We showed that there is a one-to-one correspondence between $(m,n)$-rank words $w$ and $(m,n)$-Dyck paths $\Pi(w)$.  Furthermore, since $\area(w)=\area(\Pi(w))$ and $\dinv(w)=\dinv(\Pi(w))$,
\begin{align} W_{m,n}(1,q,t) &= \sum_w q^{\dinv(w)}t^{\area(w)} \\
 & = \sum_{\Pi(w)}q^{\dinv(\Pi(w))}t^{\area(\Pi(w))} \\
 & = C_{m,n}(q,t). \end{align}
\end{proof}

\section{The Case $m=3$}

When $m=3$ and $n$ is not divisible by 3, rank words can be used to create a correspondence between certain integer triples and $(3,n)$-Dyck paths.  Note that for any $(3,n)$-Dyck path there are no positive ranks in the third column.  Further, by Corollary \ref{rightcol}, any cell in the second column that is above a $(3,n)$-Dyck path $\pi$ lies in $\Dinv(\pi)$.

\begin{prop} \label{classifyprop}
Let $\pi$ be a $(3,n)$-Dyck path.  Any cell $(1,v)$ in the first column and above $\pi$ that is in $\Skips(\pi)$ must satisfy one of the following:

\begin{equation} \label{cond1} \arm(1,v)=1 \text{ \qquad and\qquad } \leg(1,v)<\frac{n}{3}-1,\end{equation}
or
\begin{equation} \label{cond2} \arm(1,v)=0 \text{\qquad and \qquad}\leg(1,v)>\frac{n}{3}. \end{equation}

\end{prop}

\begin{proof}
Let $(1,v)$ be a cell in the first column above the path $\pi$ that is not in $\Dinv(\pi)$. 

Suppose $\arm(1,v)=1$.  Then 
\[ \gamma[(1,v)^\Rightarrow]<0<\gamma[(1,v)^\downarrow]. \]
So for $(1,v)\in\Skips(\pi)$ it must be that
\[ \gamma(1,v)-3\cdot (\leg(1,v)+1)=\gamma[(1,v)^\Downarrow]>\gamma[(1,v)^\rightarrow]=\gamma(1,v)-n, \]
or
\[ 3\cdot(\leg(1,v)+1)<n. \]
By solving for $\leg(1,v)$ we get the desired inequality.

Suppose that $\arm(1,v)=0$.  Then
\[ \gamma[(1,v)^\Downarrow]<\gamma(1,v)=\gamma[(1,v)^\rightarrow]. \]
So for $(1,v)\in\Skips(\pi)$ it must be that
\[ \gamma(1,v)-n=\gamma[(1,v)^\Rightarrow]>\gamma[(1,v)^\downarrow]=\gamma(1,v)-3\cdot\leg(1,v), \]
or
\[ n<3\cdot\leg(1,v). \]
By solving for $\leg(1,v)$ we get the desired inequality.
\end{proof}

These propositions tell us that the only cells above a path that are in $\Skips$ are those without a cell immediately to the right above the path with many cells below, or those with a cell immediately to the right above the path with few cells below.  Thus we can have one case or the other, but not both.  This leads us to an important specialization of $\skips$ when $m=3$. 

\begin{theorem}
For a $(3,n)$-rank word $w$, $\skips(w)$ is equal to the number of unhighlighted letters with a highlighted letter immediately to the left.
\end{theorem}
\begin{proof} What this theorem states is that for each equivalence class in $\left|\rfrac{S}{\sim}\right|$ there is a unique representative $(a_k,b_\ell)$ where $a_k$ appears immediately to the left of $b_\ell$.  Let $(c_r,d_t)$ be an ordered pair of letters in $w$ where $c_r$ is highlighted, $d_t$ is unhighlighted and $c_r$ appears to the left of $d_t$, but not necessarily {\it immediately} to the left.  Set $a_k$ to be the rightmost highlighted entry to the left of $d_t$ and set $b_\ell$ to be the leftmost unhighlighted entry to the right of $a_k$.  Note that it is possible that $a_k=c_r$ or $b_\ell=d_t$.

Recall that if a letter is unhighlighted then every letter with the same color lying to the left must also be unhighlighted.  Since the only possible colors are 1 and 2 it is immediate that $k=r$ and $\ell=t$.  Therefore,
\begin{align*} (c_k,d_\ell) & \sim (c_k,b_\ell) \\
 & \sim (a_k,b_\ell)\;. \end{align*}
This shows that every equivalence class contains a pair of letters that are adjacent in the rank word.

Suppose that $(a_k,b_\ell)\sim (c_r,d_t)$ where $a_k$ is immediately left of $b_\ell$ and $c_r$ is immediately left of $d_t$ in the rank word.  For these ordered pairs to be equivalent, it must be that either $k=r$ or $\ell=t$.  Since the only available colors are 1 and 2 and $r\neq t, k\neq\ell$, it must follow that $k=r$ and $\ell=t$.  From the equivalence relation it must be true that either 
\[ a \equiv c \pmod n \qquad \text{or} \qquad b\equiv d \pmod n. \]

If $b\equiv d \pmod n$ then in the corresponding Dyck path the cells containing ranks $b$ and $d$ must be in the same row.  However, $\ell=t$ so the cells are also in the same column.  Therefore $b=d$.  Since $b_\ell=d_t$, and $a_k$ is immediately left of $b_\ell$ and $c_r$ is immediately left of $d_t$ it follows that $a_k=c_r$.  

If $a\equiv c \pmod n$ then the result is the same.  So $\skips(w)$ is the number of unhighlighted letters in $w$ with a highlighted letter immediately to the left.
\end{proof}

\subsection{The rank word construction algorithm}

A triple of non-negative integers $(a,d,s)$ is said to be a {\it Dyck triple} if $s\leq \min{\{a,d\}}$ and $a+d+s+1$ is not divisible by 3. From a Dyck triple, by setting $n=a+d+s+1$, we will describe how to construct a unique $(3,n)$-rank word and from this rank word we can construct a unique Dyck path, as originally presented in \cite{Kalis15}.

\vspace{3mm}
\framebox[15.8cm]{
\begin{algorithm}[H]
\SetAlgoLined
\KwData{A Dyck triple $(a,d,s)$}
\KwResult{Construct a $(3,n)$-rank word where $n=a+d+s+1$. }
Write the empty $(3,n)$-word and highlight the rightmost (largest) $d$ ranks in the rank word.\\
\For{$i=1$ to $s$}{
Let $r_c$ be the smallest unhighlighted rank with a highlighted rank immediately to the \\ 
 right.  Find the largest rank that is less than $r$ with color not equal to $c$ and highlight it. 
}
\end{algorithm}
}
\begin{center}The rank word construction algorithm\end{center}

\begin{example} Suppose that we wish to construct a $(3,8)$-rank word from integers 
\[a=3,\qquad d=2, \qquad s=2. \]  First we write the ranks and color them:
\[ \rankword{3}{8}{0,8,16}{6cm} \;. \]
We highlight the largest $d=2$ ranks:
\[ \rankword{3}{8}{0,8,10}{6cm} \;. \]
The smallest unhighlighted rank with a highlighted rank immediately to the right is $7_1$ so we highlight the next largest rank with color $2$:
\[ \rankword{3}{8}{0,5,10}{6cm} \;. \]
Since $s=2$ we repeat this process one more time with the selected rank $4_1$:
\[ \rankword{3}{8}{0,2,10}{6cm} \;. \]
\end{example}

\begin{example} The $(3,8)$-rank word constructed from the integers
\[ a=5,\qquad d=1,\qquad s=1 \]
is 
\[ \rankword{3}{8}{0,5,13}{6cm} \;. \]
\end{example}

\begin{prop} Given a Dyck triple $(a,d,s)$ set $n=a+d+s+1$.  The rank word construction algorithm will give a valid $(3,n)$-rank word.
\end{prop}

\begin{proof}  We need to show that there are enough ranks to perform $s$ skips.  A $(3,n)$-word consists of all of the positive ranks of the form $mn-un-3v$ where $u,v\geq 1$.  So the largest $\lceil n/3 \rceil$ ranks have color $1$ and then the ranks alternate colors.

If $d\geq\lceil n/3 \rceil$ then each skip corresponds to two entries, one skipped and one not skipped.  So we need $2s+d\leq s+a+d$, which is true.  If $d<\lceil n/3 \rceil$ then the first skip and the rightmost $d$ entries together encompass $\lceil n/3 \rceil+1$ of the total entries.  The remaining skips then correspond to two entries. Again, we need $2(s-1)+\lceil n/3 \rceil+1\leq s+a+d$, which is equivalent to
\begin{equation}\label{eqstar} s-1+\lceil n/3 \rceil \leq a+d.\end{equation}
Recall that area is the unhighlighted entries.  Since there were $\lceil n/3\rceil -d$ skipped entries at the first skip and $s-1$ skipped entries in the middle,  $a\geq \lceil n/3\rceil -d + s-1$, which is exactly inequality \eqref{eqstar}.
\end{proof}

\begin{cor}
If $(a_1,d_1,s_1)$ and $(a_2,d_2,s_2)$ are distinct Dyck triples with constructed rank words $w_1$ and $w_2,$ respectively, then $w_1\neq w_2.$
\end{cor}

\begin{proof}
If $d_1\neq d_2$, then the rank word construction algorithm would highlight different numbers of letters on the right, so $w_1\neq w_2$.  If $s_1\neq s_2$ then the number of unhighlighted letters with highlighted letters to the immediate left would be different, so $w_1\neq w_2$.  If $a_1\neq a_2$ then the number of unhighlighted letters would be different, so  $w_1\neq w_2$.
\end{proof}

\subsection{Correspondence of $(3,n)$-rank words}

We have shown that a Dyck triple $(a,d,s)$ with $n=a+d+s+1$ corresponds to a unique $(3,n)$-rank word and thus to a unique $(3,n)$-Dyck path.  Now we will address the converse and show that every $(3,n)$-Dyck path corresponds to a triple of non-negative integers $(a,d,s)$ with $s\leq\min{\{a,d\}}$ and $a+d+s+1=n$.

\begin{prop}
If $w$ is a $(3,n)$-rank word then $\skips(w)\leq\dinv(w)$ and $\skips(w)\leq\area(w)$.
\end{prop}
\begin{proof}
Suppose that $\pi=\pi(w)$ is the $(3,n)$-Dyck path corresponding to $w$.  If $\skips(w)=0$ then the result is obvious, so assume that $\skips(w)>0$.  
Since each $\skips$ corresponds to an unhighlighted letter in $w$ with a highlighted letter immediately to the left, there must be at least as many unhighlighted letters in $w$ as $\skips$.  Thus $\skips(w)\leq\area(w)$.

If $r_c$ is an unhighlighted letter in $w$ with a highlighted letter to the immediate left then every unhighlighted letter with a highlighted letter to the left must be colored $c$.

\emph{Case 1: Suppose that $c=2$}.  Every rank of color $1$ that lies to the right of $r_c$ is highlighted.  Thus the $\lceil n/3\rceil$ rightmost ranks, which are all colored 1, must be highlighted.  Since the number of cells of color 2 is $\lfloor n/3 \rfloor<\lceil n/3\rceil \leq d$ it must follow that $\skips(\pi)<\dinv(\pi)$.

\emph{Case 2:  Suppose that $c=1$}.  Then every unhighlighted letter with a highlighted letter immediately to the left must be colored 1 and the highlighted letter immediately to the left must be colored 2.  Thus $\skips$ corresponds to the number of letters of color 2 appearing to the left of unhighlighted letters of color 1.  Each such letter of color 2 must correspond to a cell in column 2 of $\pi$ that is above the path.  But these cells are in $\Dinv(\pi)$.  So $\skips(w)\leq\dinv(w)$.
\end{proof}

This leads us to the following proposition:
\begin{prop} Suppose that $\pi$ is a $(3,n)$-Dyck path and set 
\[ a=\area(\pi), \qquad d=\dinv(\pi), \qquad s=\skips(\pi). \]
The rank word obtained from the Dyck triple $(a,d,s)$ is the rank word of $\pi$.
\end{prop}
\begin{proof}
To construct a $(3,n)$-Dyck path it is equivalent to choose a number of cells in the first column as well as a number of cells in the second column (subject to constraints) to be above the path.  This is essentially choosing a number of entries colored 1 as well as a number of entries colored 2 in the rank word to highlight.  Moreover, all of the highlighted entries of a specific color must be the rightmost entries, i.e. if $r_c$ is a highlighted rank and $t_c$ is another rank with $t>r$, then $t_c$ must be highlighted as well.  Suppose that there are $k$ cells above the path in the first column and $\ell$ cells above the path in the second column.

\emph{Case 1:  Suppose that $k<n/3$.}  We will show that $k=\dinv(\pi)$ and $\ell=\skips(\pi)$.  From Proposition \ref{classifyprop}, each cell in column one that has $\arm$ equal to zero will be in $\Dinv(\pi)$ because the $\leg$ is certainly less than $n/3$.  Every cell in column one that has $\arm$ equal to one is in $\Skips(\pi)$.  

Suppose $\ell>0$. Consider the upper left cell $(1,n)$. We have $\arm(1,n)=1$ and  $\leg(1,n)=k-1 <\frac{n}{3}-1$. From Proposition \ref{classifyprop}, $(1,n)\in\Skips(\pi).$  The other cells in column one with arm one have shorter legs so they will also lie in $\Skips(\pi)$.

In the rank word constructed from $(a,d,s)$ the rightmost $\dinv(\pi)=k$ entries will be highlighted, and they are all color 1.  Then the remaining rightmost 1-colored entries will be skipped and the rightmost $\skips(\pi)=\ell$ 2-colored entries will be highlighted.  This concludes the proof of this case.

\emph{Case 2:  Suppose that $k>n/3$ and $k-\ell >\frac{n}{3}+1$.}  From Proposition \ref{classifyprop},  $\skips(\pi)$ is equal to the number of cells with $\arm$ zero and whose $\leg$ is greater than $n/3$.  There are $k-\ell-\lceil n/3 \rceil$ such cells so $\skips(\pi)=k-\ell-\lceil n/3\rceil$ and $\dinv(\pi)=2\ell+\lceil n/3 \rceil$.

In the rank word constructed from $(a,d,s)$ there are $\lceil n/3\rceil$ entries with color 1 on the right and then the entries alternate in color, beginning with 2.  Thus we will highlight $\lceil n/3 \rceil$ color 1 entries, then $\ell$ entries with color 2 and $\ell$ entries with color 1.  Since the last entry we highlighted was 1, we will begin skipping color 2 entries and highlight $k-\ell-\lceil n/3\rceil$ color 1 entries.  Therefore the rightmost $\ell$ color 2 entries are highlighted and the rightmost $\lceil n/3\rceil+\ell+\left(k-\ell-\lceil n/3\rceil\right)=k$ color 1 entries are highlighted.  This completes the proof for this case.

\emph{Case 3:  Suppose that $k>n/3$ and $k-\ell<\frac{n}{3}+1$.}  From Proposition \ref{classifyprop},  $\skips(\pi)$ is equal to the number of cells with $\arm$ equal to one and whose $\leg$ is less than $\frac{n}{3}-1$.  There are $\ell-k+\lfloor n/3\rfloor$ of these.  Thus $\skips(\pi)=\ell-k+\lfloor n/3\rfloor$ and $\dinv(\pi)=2k-\lfloor n/3\rfloor$.

In the rank word constructed from $(a,d,s)$ there are $\lceil n/3\rceil$ color 1 entries on the right and then the entries alternate in color, beginning with 2.  Thus we will highlight $\lceil n/3\rceil$ color 1 entries.  Then we will highlight $2k-\lfloor n/3\rfloor - \lceil n/3\rceil=2k-2\cdot\lfloor n/3\rfloor -1$ entries: $k-\lfloor n/3 \rfloor$ color 2 entires and $k-\lfloor n/3 \rfloor -1$ color 1 ones. Then we will highlight $\ell-k+\lfloor n/3\rfloor$ color 2 entries at the same time skip $\ell-k+\lfloor n/3\rfloor$ color 1 entries. So we have highlighted $\lceil n/3 \rceil + k-\lfloor n/3\rfloor -1=k$ entries with color 1 and $k-\lfloor n/3\rfloor + \ell-k+\lfloor n/3 \rfloor=\ell$ entries with color 2.  This completes the proof.
\end{proof}

Therefore, distinct Dyck triples will construct distinct $(3,n)$-rank words and therefore distinct $(3,n)$-Dyck paths.  Similarly, if $\pi,\psi$ are three-column Dyck paths such that
\begin{align*} \area(\pi)&=\area(\psi), \\
\dinv(\pi)&=\dinv(\psi), \\
\skips(\pi)&=\skips(\psi), \end{align*}
then they would give the same Dyck triple, and therefore the same rank word, and so they would have to be identical Dyck paths.

\begin{remark} Suppose that $\pi$ is a $(3,n)$-Dyck path.  Since $\skips(\pi)\leq \area(\pi)$ and $\skips(\pi)\leq\dinv(\pi)$, the maximum value for $\skips(\pi)$ occurs when $\area(\pi)$ and $\dinv(\pi)$ are close together, yet minimal.  Thus $\skips(\pi)\leq \lfloor n/3\rfloor$.
\end{remark}

\subsection{Generating functions when $m=3$}

\begin{cor} Suppose that $n>1$ is an integer that isn't divisible by $3$.  Then
\[ W_{3,n}(b,q,t)=\sum_{i=0}^{\lfloor n/3 \rfloor} b^i\cdot s_{n-1-2i,i}(q,t), \]
where $\{s_\lambda\}_\lambda$ is the Schur basis of symmetric functions.
\end{cor}
\begin{proof} For each $i=0,\ldots,\lfloor n/3\rfloor$, let
\begin{align} a_j(i) & =n-1-2i-j\\
d_j(i) &=i+j. \end{align}
Then for each $0\leq j \leq n-1-3i$, the triple of non-negative integers $(a_j(i),d_j(i),i)$ satisfies
\[ a_j(i)+d_j(i)+i=n-1, \qquad \text{and}\qquad i\leq \min{\{a_j(i),d_j(i)\}}, \]
so it is a Dyck triple.
Therefore,
\begin{align} W_{3,n}(b,q,t)&=  \sum_{i=0}^{\lfloor n/3 \rfloor} b^i\cdot \sum_{j=0}^{n-1-3i} q^{d_j(i)}t^{a_j(i)} \\
& = \sum_{i=0}^{\lfloor n/3 \rfloor} b^i\cdot \sum_{j=0}^{n-1-3i} q^{i+j}t^{n-1-2i-j}. \end{align}
But $ \sum_{j=0}^{n-1-3i} q^{i+j}t^{n-1-2i-j}=s_{n-1-2i,i}(q,t)$.  This completes the proof.
\end{proof}

\begin{cor} Suppose that $n>1$ is an integer that isn't divisible by 3.  The $(3,n)$-rational $q,t$-Catalan polynomial is symmetric in $q$ and $t$.
\end{cor}
\begin{proof}
Since Schur functions are symmetric functions,
\[ W_{3,n}(b,q,t)=\sum_{i=0}^{\lfloor n/3 \rfloor} b^i\cdot s_{n-1-2i,i}(q,t)=\sum_{i=0}^{\lfloor n/3 \rfloor} b^i\cdot s_{n-1-2i,i}(t,q) = W_{3,n}(b,t,q). \]
Theorem \ref{wordtocatalan} shows that $W_{3,n}(1,q,t)=C_{3,n}(q,t)$.  So
\[ C_{3,n}(q,t)=W_{3,n}(1,q,t)=W_{3,n}(1,t,q)=C_{3,n}(t,q). \]
\end{proof}

To construct a bijection on the set of $(3,n)$-Dyck paths that exchanges $\dinv$ with $\area$:
\begin{enumerate}
\item Let $\pi$ be a $(3,n)$-Dyck path.
\item Find the statistics $(\area(\pi),\dinv(\pi),\skips(\pi))$.
\item Use the $(3,n)$-rank word construction algorithm to find the $(3,n)$-rank word $w$ with Dyck triple
\[ a=\dinv(\pi), \qquad d=\area(\pi), \qquad\text{and}\qquad s=\skips(\pi)\;. \]
\item Find the unique $(3,n)$-Dyck path corresponding to $w$, $\tau=\Pi(w)$.
\end{enumerate}

\section{Acknowledgements}
We would like to thank Adriano Garsia and Jim Haglund for the motivation to work
on this problem. We wish to thank Emily Leven for inspiring conversations on this topic.  We are also grateful to the anonymous reviewers at the Formal Power Series and Algebraic Combinatorics 2015 conference for providing useful comments and papers to compare to our work.

\bibliographystyle{alpha}
\bibliography{catalanbib}

\newcommand{\etalchar}[1]{$^{#1}$}
\begin{thebibliography}{HHL{\etalchar{+}}05}

\bibitem[GH96]{Garsia96}
A.~M. Garsia and M.~Haiman.
\newblock Some natural bigraded {$S_n$}-modules and {$q,t$}-{K}ostka
  coefficients.
\newblock {\em Electron. J. Combin.}, 3(2):Research Paper 24, approx.\ 60 pp.\,
  1996.
\newblock The Foata Festschrift.

\bibitem[GH02]{Garsia02}
A.~M. Garsia and J.~Haglund.
\newblock A proof of the {$q,t$}-{C}atalan positivity conjecture.
\newblock {\em Discrete Math.}, 256(3):677--717, 2002.
\newblock LaCIM 2000 Conference on Combinatorics, Computer Science and
  Applications (Montreal, QC).

\bibitem[GM14]{Gorsky14}
E.~Gorsky and M.~Mazin.
\newblock Compactified {J}acobians and {$q,t$}-{C}atalan numbers, {II}.
\newblock {\em J. Algebraic Combin.}, 39(1):153--186, 2014.

\bibitem[Hag03]{Haglund03}
J.~Haglund.
\newblock Conjectured statistics for the {$q,t$}-{C}atalan numbers.
\newblock {\em Adv. Math.}, 175(2):319--334, 2003.

\bibitem[Hag08]{Haglund08}
J.~Haglund.
\newblock {\em The {$q$},{$t$}-{C}atalan numbers and the space of diagonal
  harmonics}, volume~41 of {\em University Lecture Series}.
\newblock American Mathematical Society, Providence, RI, 2008.
\newblock With an appendix on the combinatorics of Macdonald polynomials.

\bibitem[HHL{\etalchar{+}}05]{Haglund05}
J.~Haglund, M.~Haiman, N.~Loehr, J.~B. Remmel, and A.~Ulyanov.
\newblock A combinatorial formula for the character of the diagonal
  coinvariants.
\newblock {\em Duke Math. J.}, 126(2):195--232, 2005.

\bibitem[{Hik}12]{Hikita12}
T.~{Hikita}.
\newblock {Affine Springer fibers of type $A$ and combinatorics of diagonal
  coinvariants}.
\newblock {\em ArXiv e-prints}, March 2012.

\bibitem[KL15]{Kalis15}
R.~{Kaliszewski} and H.~Li.
\newblock {The $(m, n)$-rational $q, t$-Catalan polynomials for $m = 3$ and
  their $q, t$-symmetry}.
\newblock {\em DMTCS FPSAC Proceedings}, pages 769--780, 2015.

\bibitem[{Lev}14]{Leven14}
E.~{Leven}.
\newblock {Two special cases of the Rational Shuffle Conjecture}.
\newblock {\em DMTCS FPSAC Proceedings}, pages 789--800, 2014.

\bibitem[LLL14]{LLL14}
K.~Lee, L.~Li, and N.~A. Loehr.
\newblock Combinatorics of certain higher {$q,t$}-{C}atalan polynomials:
  chains, joint symmetry, and the {G}arsia-{H}aiman formula.
\newblock {\em J. Algebraic Combin.}, 39(4):749--781, 2014.

\end{thebibliography}

\end{document}